\documentclass[12pt,a4paper,psamsfonts]{amsart}
\usepackage[english]{babel}
\usepackage[T1]{fontenc}
\usepackage{fancyhdr}
\usepackage{appendix}
\usepackage{amssymb,amscd,amsxtra,calc}
\usepackage{mathrsfs}
\usepackage{cmmib57}
\usepackage{multirow}
\usepackage[all]{xy}
\usepackage{longtable}
\usepackage[colorlinks=true,linkcolor=blue,anchorcolor=blue,citecolor=blue,pagebackref,linktocpage]{hyperref}
\usepackage{cleveref}
\usepackage{tikz}
\usepackage{cite}
\theoremstyle{plain}
    \newtheorem{thm}{Theorem}[section]

     \newtheorem{conjecture}[thm]{Conjecture}
    \newtheorem{corollary}[thm]{Corollary}

    \newtheorem{proposition}[thm]{Proposition}
    \newtheorem{question}[thm]{Question}
    \newtheorem{theorem}[thm]{Theorem}

\theoremstyle{definition}

    \newtheorem*{notation*}{Notation and Terminology}
      
    \newtheorem{remark}[thm]{Remark}
    
\theoremstyle{remark}

\newcommand{\arxiv}[1]{\href{https://arxiv.org/abs/#1}{{\tt arXiv:#1}}}

\newcommand{\bP}{\mathbb{P}}
\newcommand{\bQ}{\mathbb{Q}}

\newcommand{\bC}{\mathbb{C}}

\newcommand{\bZ}{\mathbb{Z}}

\newcommand{\mstriangle}[1]{
\begin{tikzpicture}[x=0.3cm,y=0.3cm]
\draw (-0.4,-0.433) -- (1.4,-0.433);
\draw (-0.2,-0.7794) -- (0.7,0.7794);
\draw (1.2,-0.7794) -- (0.3,0.7794);
\end{tikzpicture}
}
\newcommand{\mssharp}[1]{
\begin{tikzpicture}[x=0.3cm,y=0.3cm]
\draw (-0.8,-0.5) -- (0.8,-0.5);
\draw (-0.8,0.5) -- (0.8,0.5);
\draw (-0.5,-0.8) -- (-0.5,0.8);
\draw (0.5,-0.8) -- (0.5,0.8);
\end{tikzpicture}
}

\makeatletter

\newcommand{\Rmnum}[1]{\expandafter\@slowromancap\romannumeral #1@}
\makeatother

\begin{document}

\title[Integral Zariski decompositions]{Integral Zariski decompositions on smooth projective surfaces}

\author{Sichen Li}
\address{
School of Mathematics, East China University of Science and Technology, Shanghai 200237, P. R. China}
\email{\href{mailto:sichenli@ecust.edu.cn}{sichenli@ecust.edu.cn}}
\begin{abstract}
In this paper, we characterize smooth projective surfaces on which every integral pseudoeffective divisor has an integral Zariski decomposition.
\end{abstract}
\thanks{The research is supported by the Shanghai Sailing Program (No. 23YF1409300). }

\maketitle
\section{Introduction}
The bounded negativity conjecture (BNC for short)  is one of the most intriguing problems in the theory of projective surfaces and can be formulated as follows:
\begin{conjecture}	
\cite[Conjecture 1.1]{Bauer et al 2013}
\label{BNC}
For a smooth projective surface $X$ over $\bC$ there exists an integer $b(X)\ge0$ such that $C^2\ge-b(X)$ for every curve $C\subseteq X$.
\end{conjecture}
The geometric significance of Zariski-Fujita decomposition \cite{Fujita79, Zariski62} lies in the fact that, given an integral pseudo-effective divisor $D$ on  $X$ with Zariski decomposition $D=P+N$ such that they are  $\bQ$-divisors,  $P$ is nef, and one has for every sufficiently divisible integer $m\ge1$ the equality
$$
H^0(X,\mathcal O_X(mD))=H^0(\mathcal O_X(mP)).
$$ 
We say $X$ has {\it bounded Zariski denominators} if there exists such $d(X)=m$ which is independent of the choice of $D$.
In fact, $d(X)$ is defined as the maximum of the denominators appearing in the Zariski decomposition of all integral pseudo-effective divisor on $X$.

A celebrated theorem of Bauer, Pokora, and Schmitz \cite{BPS17} says that $X$ has bounded Zariski denominators if and only if $X$ satisfies the BNC as follows.
\begin{theorem}
For a smooth projective surface $X$ over an algebraically closed field the following two statements are equivalent:
\begin{itemize}
\item[(1)] $X$ has bounded Zariski denominators.
\item[(2)]  $X$ satisfies the BNC.
\end{itemize}
\end{theorem}
\begin{remark}
To the best of our knowledge, we summarize the   known cases of BNC as follows:
\begin{enumerate}
	\item $X$ admits a surjective endomorphism that is not an isomorphism (cf. \cite[Proposition 2.1]{Bauer et al 2013}).
	\item The effective cone is polyhedral (cf. \cite[Proposition 1.1]{AL11}).
	\item $X$ satisfies the bounded cohomology property, i.e., there is a constant $c_X>0$ such that $h^1(\mathcal O_X(C))\le h^0(\mathcal O_X(C))$ for every curve $C$ on $X$ (cf. \cite[Proposition 14]{Ciliberto et al 2017}, \cite{Li21, Li23a, Li23b}).
\item  $X$ has $d(X)=1$ (cf. \cite[Theorem 2.3]{BPS17}, \cite{HPT15, Li19}).
The author of \cite{Li19} characterized smooth projective surfaces $X$ with the Picard number $\rho(X)=2$ and $d(X)=1$.
\end{enumerate}
\end{remark} 
It is well-known that the following SHGH conjecture implies Nagata's conjecture (cf. \cite[p. 772]{Nagata59}), which is motivated by Hilbert's 14th problem.
\begin{conjecture}
(cf. \cite[Conjecture 1.1, 2.3]{CHMR13})
Let $X$ be a composite of blow-ups of $\bP^n$ at points $p_1,\cdots, p_n$ in a very general position.
Then every negative curve on $X$ is a (-1)-rational curve.
\end{conjecture}
\begin{proposition}
\cite[Proposition 1.8]{Li19}
Let $X$ be a composite of blow-up of $\bP^n$ at   $n$ distinct points.
Then every negative curve $C$ has $C^2=-1$ if and only if $X$ has $d(X)=1$.
\end{proposition}
It is crucially important to classify smooth projective surfaces $X$ with $d(X)=1$ since it satisfies the BNC and will be used to solve the SHGH conjecture.

Below is a numerical characterization of smooth projective surfaces $X$ with $d(X)=1$.
\begin{theorem}
\label{MainThm}
(cf. Theorem \ref{Main-Thm})
Let $X$ be a smooth projective surface. 
Then every pseudo-effective integral divisor on $X$ has an integral Zariski decomposition if and only if the following statements hold.
\begin{itemize}
	\item[(a)] $C^2|(C\cdot D)$ for every negative curve $C$ and every curve $D$.
	\item[(b)] for any two curves $C_1$ and $C_2$, if their intersection matrix is negative definite, then $(C_1\cdot C_2)=0$. 
\end{itemize}	
\end{theorem}
\begin{remark}
In \cite{Bauer09}, Bauer demystified Zariski decompositions by reducing the whole process to basic linear algebra.
Indeed, the proof of Theorem  \ref{MainThm} is also the process of basic linear algebra.
\end{remark}
By Theorem \ref{MainThm}, to find more examples of smooth projective surfaces with $d(X)=1$, it seems very mysterious and highly constrained, but the SHGH conjecture points us toward one of the ultimate goals.
\begin{corollary}
Let $X$ be a composite of blow-ups of $\bP^n$ at $n$ distinct points.
If there is a negative curve $C$ and another curve $D$ such that $C^2 \nmid (C\cdot D)$, then the SHGH conjecture fails.
\end{corollary}
Finally, we show that every fiber of smooth projective fibered surfaces $X$ with $d(X)=1$ is irreducible as follows.
\begin{proposition}
\label{MainProp}
Let $X$ be a smooth projective surface with $d(X)=1$.
If $X$ admits a fibration $\pi: X\to B$ over a curve $B$, then every fiber is irreducible.
\end{proposition}
The paper is organized as follows.
Theorem \ref{MainThm} and Proposition \ref{MainProp} are proved in section \ref{The proof}.
In section \ref{application}, we collect some results of smooth projective surfaces $X$ with $d(X)=1$ which is motivated by \cite[Theorem 1.6]{Li19}.
In the last section, we give a remark on a question of Bauer et al \cite{Bauer et al 2013}.

{\bf Notation and Terminology.}
Let $X$ be a smooth projective surface over $\bC$.
\begin{itemize}
\item By a curve on $X$ we will mean a reduced and irreducible curve.
\item A negative curve on $X$ is a curve with negative self-intersection.
\item A prime divisor $C$ on $X$ is either a nef curve or a negative curve in which case that $h^0(\mathcal O_X(C))=1$.
\item A curve $C$ on $X$ is nef if $C\cdot P\ge0$ for every curve $P$ on $X$.
\item An integral pseudoeffective divisor $D$ on $X$ is linear equivalent to $\sum a_iC_i$ with each $a_i\in\bZ$ and each $C_i$ is a curve on $X$ and $D\cdot C\ge0$ for every nef curve $C$ on $X$.
\end{itemize}
{\bf Acknowledgments.} 
The author would like to thank Yi Gu, Feng Hao, and Xin L\"u for asking questions.
\section{A numerical characterization}
\label{The proof}
\begin{proposition}
\label{KeyProp}
Let $X$ be a smooth projective surface with $d(X) = 1$. Let $C_1, . . . , C_k$
be some curves on $X$ and $I(C_1, . . . , C_k)$ their intersection matrix. If $I(C_1,\cdots, C_k)$ is negative definite, then $I(C_1,\cdots, C_k)$ is a diagonal matrix. As a result, for every negative curve $C$ and every curve $D,$ we always have $C^2|(C\cdot D)$.	
\end{proposition}
\begin{proof}
To prove that $I(C_1, . . . , C_k)$ is a diagonal matrix, it suffices to show that $(C_i\cdot C_j ) =0$ for $i\ne j\in\{1,\cdots, k\}$.
Now assume that $I(C_i, C_j)$ is negative definite. 
Take a very ample divisor $H$ on $X$.
And take a curve $D\in |kH+C|$ for $k\gg 0$.
Note that the intersection matrix of $C_i$ and either $D$ or $H$ is not negative definite.
Then by \cite[Lemma 2.3]{Li19}, $C_i^2|(k(H\cdot C_i)+(C_i\cdot C_j))$ and $C_i^2|(H\cdot C_i)$.
So $C_i^2|(C_i\cdot C_j)$.
We also have $C_j^2|(C_i\cdot C_j)$ since $I(C_i, C_j)$ is negative definite.
If $(C_i\cdot C_j)>0$, then $(C_i^2)(C_j^2)-(C_i\cdot C_j)^2<0$.
But $I(C_i, C_j)$ is negative definite.
This leads to a contradiction.
So $(C_i\cdot C_j)=0$.
The last statement follows from \cite[Lemma 2.3]{Li19}.
\end{proof}
\begin{theorem}
\label{Main-Thm}
Let $X$ be a smooth projective surface. 
Then $X$ has $d(X)=1$ if and only if the following statements hold.
\begin{itemize}
	\item[(a)] $C^2|(C\cdot D)$ for every negative curve $C$ and every curve $D$.
	\item[(b)] If the intersection matrix of any two curves $C_1$ and $C_2$ is negative definite, then $(C_1\cdot C_2)=0$. 
\end{itemize}	
\end{theorem}
\begin{proof}
Suppose $d(X)=1$.
Then the proof follows from Proposition \ref{KeyProp}.
Conversely, assume that (a) and (b) hold.
Let the pseudoeffective integral divisor $D$ be written as $\sum b_iD_i$ with each $b_i\in \bZ$ and each $D_i$ is a curve on $X$.
The Zariski-Fujita decomposition states that $D$ can be written uniquely as a sum
$$
       D=P+N
$$
of $\bQ$-divisors such that 
\begin{enumerate}
\item[(i)] $P$ is nef,
\item[(ii)] $N$ is effective and has negative definite intersection matrix if $N\ne0$.
\item[(iii)] $P\cdot C=0$ for every component $C$ of $N$.
\end{enumerate}
Let $N=\sum_{i=1}^k a_iC_i$ with $(D\cdot C_i)<0$ for each $i\in\{1,\cdots, k\}$.
They are given as the unique solution of the system of equations
$$
 D\cdot C_j=(P+\sum_{i=1}^k a_iC_i)\cdot C_j=a_jC_j^2 \text{ for all }  j \in\{1,\cdots, k\}.
$$
Then $a_j=\frac{(D\cdot C_j)}{C_j^2}\in\bZ_{>0}$ for all $j\in \{1,\cdots, k\}$.
As a result, $d(X)=1$.
\end{proof}
\begin{proof}[Proof of Proposition \ref{MainProp}]
Let $F$ be a fiber of $\pi$.
We first assume that $F=a_1C_1+a_2C_2$ with $a_1, a_2>0$.
By Zariski's lemma, the intersection matrix of $C_1$ and $C_2$ is negative semi-definite, $C_1^2<0$ and $C_2^2<0$.
Then $(C_1^2\cdot C_2^2)-(C_1\cdot C_2)^2\ge0$ and $(C_1\cdot C_2)\ge\max\{-C_1^2,-C_2^2\}$ by Theorem \ref{Main-Thm} (a).
So $(C_1\cdot C_2)=-C_1^2=-C_2^2>0$.
Then $F^2=(a_1+a_2)^2(C_1\cdot C_2)>0$.
This leads to a contradiction.
Now assume $F=\sum_{i\in I} a_iC_i$ has at most three components.
Then by Zariski's lemma, we know that $I(C_1, C_2)$ is negative definite for $i\ne j\in I$.
So by Proposition \ref{KeyProp}, $(C_i\cdot C_j)=0$ for $i\ne j\in I$.
Then $F^2=\sum a_i^2C_i^2<0$.
This leads to a contradiction.
\end{proof}
\section{Serveral auxiliary results}
\label{application}
Let $X$ be a smooth projective surface.
It is well-known that $d(X)=1$ if every curve $C$ has $C^2\ge-1$ (cf. \cite{BPS17} and \cite[Proposition 1.2]{HPT15}).
Conversely, if $d(X)=1$, then is every negative curve a (-1)-curve?
The authors of \cite{HPT15} disproved this question by giving an example of a K3 surface $X$ with $d(X)=1, \rho(X)=2$ and two (-2)-rational curves (cf. see \cite[Claim 2.12]{Li19} for a related result).
\begin{question}
(cf. \cite[Question 2.15]{Li19})
Is there is a positive constant $\ell$ such that for any minimal smooth projective surface $X$ with $d(X)=1$ and every negative curve $C$ on $X$ which $C^2\ge-\ell$?
\end{question}
\begin{proposition}
If $X$ is a relatively minimal smooth projective surface with $\kappa(X)\le0$ and $d(X)=1$, then every negative curve $C$ on $X$ has $C^2\ge-2$.		
\end{proposition}
\begin{proof}
Let $C$ be a curve on $X$.
We first assume that $\kappa(X)=-\infty$.
Then by \cite[Claim 2.10]{Li19}, $X$ is either a ruled surface with invariant $e=1$ or a one point blow-up of $\bP^2$.
Notice that $X$ is a ruled surface with invariant $e=1$ since $X$ is relatively minimal.
This implies that $C^2\ge-1$.
Now assume that $\kappa(X)=0$.
Then $C^2\ge-2$ by the adjunction formula.
\end{proof}
\begin{proposition}
Let $X$ be a minimal smooth projective surface with $\kappa(X)=1$.
If $d(X)=1$, then every fiber of $\varphi_{|mK_X|}$ is irreducible.
In particular, if every fiber is of type $mI_0$, then every negative curve $C$ on $X$ has $g(C)\ge2$.		
\end{proposition}
\begin{remark}
In \cite[Claim 2.14]{Li19}, we assume that $\rho(X)=2$.
\end{remark}
\begin{proof}
By Proposition \ref{MainProp}, it suffices to show that every negative curve $C$ on $X$ has $g(C)\ge2$ if every fiber is of type $mI_0$.
Notice that $c_2(X)=0$ (cf. \cite[Lemma 2.13]{Li19}).
Then by \cite[Theorem 2.4]{Bauer et al 2013}, we have the following inequality:
$$
0<-C^2\le 2g(C)-2.
$$
So $g(C)\ge2$.
\end{proof}
\begin{proposition}
Let $X$ be a minimal smooth projective surface of general type with $\rho(X)=2$.
If $K_X^2\ge10$ and there is a negative curve $C$ such that $(K_X\cdot C)^2-K_X^2\cdot C^2$ is square-free, then the bicanonical map of $X$ is birational.	
\end{proposition}
\begin{proof}
By \cite[Theorem VII.7.4]{BHPV04}, the bicanonical map 
$$
  f_2: X\to \bP^{h^0(\mathcal O_X(2K_X))-1}
$$
is a morphism.
If $f_2$ is not birational and $K_X^2\ge10$, then by \cite[Theorem VII.5.1]{BHPV04}, there exists a morphism $\pi: X\to B$ whose general fiber $F$ is of genus two.
Assume $X$ has a negative curve $C$ such that $(K_X\cdot C)^2-K_X^2\cdot C^2$ is square-free.
Notice that $(F\cdot C), (F\cdot K_X)\in\bZ_{\ge0}$.
Since $\rho(X)=2$ and $K_X$ is ample, then there exists a sufficiently divisible integer $N$ such that
$$
   NF\equiv m_1K_X+m_2C, m_1, m_2\in\bZ.       
$$
Note that $F^2=0$ implies that 
$$
         K_X^2m_1^2+2(K_X\cdot C)m_1m_2+C^2m_2=0, f(\frac{m_1}{m_2}):=K_X^2(\frac{m_1}{m_2})^2+2(K_X\cdot C)\frac{m_1}{m_2}+C^2=0.
$$
Here, 
$$
         \sqrt{\Delta_f}=\sqrt{(K_X\cdot C)^2-K_X^2\cdot C^2}\in\bZ_{\ge0}.
$$
This leads to a contradiction.
\end{proof}
\section{Negative curves with higher genus}
 The well-known example of $\bP^2$ blown up at nine points shows that there are surfaces containing infinitely many (-1)-curves with genus zero.
Along similar lines, Bauer et al established the following result.
\begin{theorem}\label{Bauer-Thm}
 \cite[Theorems 4.1 and 4.3]{Bauer et al 2013} 
For either (i) each $m>0$ and $g=0$ or (ii)  each $m>1$ and each $g\ge0$, there are smooth projective surfaces containing infinitely many smooth irreducible $(-m)$-curves with genus $g$.	
\end{theorem}
Motivated by Theorem \ref{Bauer-Thm}, Bauer et al asked the following question.
\begin{question}
\cite[Question 4.4]{Bauer et al 2013}
\label{MainQue}
Is there for each $g>1$ a surface with infinitely many $(-1)$-curves of genus $g$?
\end{question}
\begin{remark}
The author showed in \cite[Proposition 2.5]{Li19} that Question \ref{MainQue} has an affirmative answer for an elliptic surface $X$ with infinitely many sections and $d(X)=1$.
\end{remark}
Let $f: X\to B$ be an elliptic surface whose generic fiber $E/K$ (an elliptic curve $E$ over the function field $K=k(B)$).
By \cite[Proposition 5.4]{SS19}, the sections of an elliptic surface $f: S\to B$ are in a natural one-to-one correspondence with $K$-rational points $P$ of $E$.
We say a section $C$ of $f: S\to B$ is non-torsion if $P$ is a non-torsion point.
As a result, $X$ has infinitely many sections if it has a non-torsion section.
This motivates us to answer Question \ref{MainQue} in some sense.
\begin{proposition}
Let $f: X\to B$ be an elliptic surface over a smooth projective curve $B$ of genus $g(B)=g$.
If $X$ has a non-torsion section $C$ and $\chi(\mathcal O_X)=1$, then $X$ has infinitely many $(-1)$-curves of fixed genus $g$.	
\end{proposition}
\begin{proof}
 By \cite[Corollaries 5.45 and  5.50]{SS19}, $C^2=-\chi(\mathcal O_X)=-1$.
Notice that $g(C)=g(B)$ as $C$ is a section of $f$.
Therefore, $X$ has infinitely many $(-1)$-curves of fixed genus $g$ as $C$ is a non-torsion section.
\end{proof}
It is unknown to us that there exists an example of elliptic surfaces $X$ over a curve $B$ of genus $g(B)=g$ such that $X$ has a non-torsion section and $\chi(\mathcal O_X)=1$.
So to completely answer Question \ref{MainQue}, we end by asking the following question.
\begin{question}
Is there  an example of an elliptic surface $X$ over a smooth projective curve $B$ of genus $g>1$ such that $X$ has a non-torsion section and $\chi(\mathcal O_X)=1$?
\end{question}

\end{document}